\documentclass[12pt]{amsart}
\usepackage{amsmath,amssymb,amsthm,mathtools}
\usepackage{mathrsfs,bm}
\usepackage{thmtools,thm-restate}
\usepackage{tikz-cd}
\usepackage{tikz}
\usepackage{enumitem}
\usepackage{graphicx}
\usetikzlibrary{matrix,arrows.meta,bending}
\usepackage{color}
\usepackage[numbers]{natbib}
\usepackage{booktabs}
\usepackage{longtable}
\usepackage{pgfplots}
\pgfplotsset{compat=1.14} 
\usetikzlibrary{arrows.meta,decorations.pathreplacing,positioning,calc} 

\usepackage[a4paper, left=3cm, right=3cm, top=3cm, bottom=3cm, footskip=15mm]{geometry}

\newtheorem{theorem}{Theorem}[section]

\newtheorem{proposition}[theorem]{Proposition}

\theoremstyle{definition}
\newtheorem{definition}[theorem]{Definition}

\theoremstyle{remark}
\newtheorem{remark}[theorem]{Remark}

\numberwithin{equation}{section}

\makeindex

\usepackage{fancyhdr}
\usepackage{calc} 

\AtBeginDocument{%
  \setlength{\textwidth}{\dimexpr\paperwidth - 6cm\relax}%
  \setlength{\oddsidemargin}{\dimexpr 3cm - 1in\relax}%
  \setlength{\evensidemargin}{\dimexpr 3cm - 1in\relax}%
  \setlength{\marginparwidth}{2.5cm}%
}

\begin{document}

\title{An inequality of multiple integral of $s$ norm of vectors in $\mathbb{R}^n$}

\author{Theophilus Agama}
\address{Department of Mathematics, Universite Laval, Quebec, Canada
}
\email{thaga1@ulaval.ca}

\subjclass[2010]{Primary 26D10; Secondary 46C05}

\date{\today}


\keywords{local product; local product space; sheet}

\begin{abstract}
In this paper, we prove some new inequalities. To facilitate this proof, we introduce the notion of \emph{local product} on a \emph{sheet} and associated space.
\end{abstract}

\maketitle

\section{Introduction}

Inequalities are among the most effective instruments in analysis because they transform a qualitative structure into a quantitative control. In the setting of inner product spaces, this philosophy is most visibly embodied by the Cauchy--Schwarz inequality, which remains a foundational estimate for both finite-dimensional and infinite-dimensional analysis; standard treatments emphasize that many of the basic estimates of Hilbert space theory are ultimately organized around this principle \cite{rudin1987}. Over time, this classical point of view has been refined in several directions. One line of development seeks sharper inner-product inequalities such as Buzano-, Kre{\i}n-, and Cauchy--Schwarz-type refinements and then uses them to derive discrete, integral, operator, and numerical-radius bounds \cite{gavrea2007,sababheh2022}. The second line expands the ambient setting itself, replacing the ordinary Hilbert structure with a semi-Hilbert or related generalized framework to recover stronger or more flexible forms of the same inequalities \cite{altwaijry2023}. More recently, there has also been sustained interest in functional inequalities on Banach spaces and in orthogonality-based characterizations of inner product spaces, reflecting the fact that the geometry of a normed or inner product space continues to support new inequality mechanisms well beyond the classical literature \cite{niculescu2024,xia2024}.\\

The present paper follows this general philosophy, but introduces a different organizing principle. Rather than starting from a single direct estimate and attempting to optimize it by repeated algebraic manipulation, we introduce a \emph{local product} on a \emph{sheet}, together with the associated local product space, as a structured way to encode multivariate integral expressions. Conceptually, the local product is a hybrid object: it combines the geometric information carried by the ambient inner product with a multidimensional integral kernel whose shape is controlled by the chosen sheet. The sheet plays the role of a scalar weight that is evaluated on the inner product and then propagated through the iterated integral. In this way, one and the same template produces several distinct inequalities by varying the sheet, while preserving a common geometric backbone. The guiding idea is therefore not to merely estimate an integral, but to build an analytic device that packages the integral and the inner-product dependence into a single object.\\

This viewpoint is particularly effective because the local product admits several canonical specializations that immediately recover the basic building blocks used later in the paper. When the sheet is the constant function, the local product collapses to the volume of the underlying rectangular region, hence producing the factor
$$
\prod_{j=1}^n\bigl(|b_j|-|a_j|\bigr).
$$
When the sheet is the logarithm, the same construction yields a logarithmically weighted local product, and this naturally leads to lower and upper bounds in which $\log(\langle a,b\rangle)$ or $\log\log(\langle a,b\rangle)$ appears as the controlling scalar term. Likewise, the identity sheet and the inverse-identity sheet produce two complementary forms of the same structure, one leading to a direct comparison with the inner product and the other to a reciprocal-type estimate. In particular, the paper shows that a single local-product mechanism can generate the three principal inequalities stated in the introduction: a lower bound involving $\log\log(\langle a,b\rangle)$, an upper bound involving $\log(\langle a,b\rangle)$, and a third lower bound involving $1/\log(\langle a,b\rangle)$. The resulting inequalities are not isolated coincidences; they arise from the same multivariate template after a careful choice of sheet and exponent.\\

This is one reason why the local product framework is of independent interest. The literature cited above shows that many modern inequality results arise by refining ambient geometry, strengthening classical estimates, or by introducing auxiliary structures such as seminorms, modified inner products, and generalized orthogonality notions \cite{gavrea2007,sababheh2022,altwaijry2023,niculescu2024,xia2024}. The present paper contributes to this landscape in a different way: it introduces an analytic template in which the role of the auxiliary object is played by a sheet, and the resulting estimate is read off from the corresponding local product. This produces a compact and flexible mechanism for manufacturing inequalities that would be difficult to obtain by a direct calculation alone. In that sense, the local product should be viewed not merely as a notation, but as a structural device that organizes the dependence of the estimate on the inner product, on the vector norms, and on the iterated integration domain.\\

The main applications obtained in the paper can be summarized as follows. First, the local product is used to produce a family of inequalities for iterated integrals over coordinate intervals in $\mathbb{R}^n$, with the integrands chosen so that the resulting bounds naturally interact with the specializations of the local product described above. Second, the method yields both upper and lower estimates depending on the chosen sheet, showing that the same framework can control a nonlinear expression from both sides. Third, the results exhibit a repeated pattern in which the factors that involve $\langle a,b\rangle$, $\|a\|$, $\|b\|$, and $\prod_{j=1}^n(|b_j|-|a_j|)$ appear in a unified way, rather than as unrelated terms introduced by separate arguments. The local product therefore provides a common language for several inequalities that, at first sight, seem to have a different analytic character.

\subsection{Organization of the paper} The structure of the paper is as follows. Section~2 introduces the local product and the associated local product space and records the principal specializations obtained from the constant, logarithmic, identity, inverse-identity, and log--log sheets. Section~3 develops the elementary properties of the local product, including basic symmetry and comparison principles that are needed later. Section~4 contains the main applications and proves the inequalities stated in the introduction by selecting the appropriate sheets and applying the structural properties established earlier. A short concluding remark is included at the end of the paper to emphasize the conceptual role of the local product as a flexible device for inequality production.\\

The following inequalities are deduced in the following paper:

\begin{theorem}
Let $\vec{a}=(a_1,a_2,\ldots,a_n),\vec{b}=(b_1,b_2,\ldots,b_n)\in \mathbb{R}^n$ be such that $e^e<\langle \vec{a},\vec{b} \rangle$ and $b_j>a_j$ for $1\leq j\leq n$. We have
\begin{align}
\int \limits_{|a_n|}^{|b_n|} \int \limits_{|a_{n-1}|}^{|b_{n-1}|}\cdots \int \limits_{|a_1|}^{|b_1|}\bigg|\log \bigg(i\frac{\sqrt[4s]{\sum \limits_{j=1}^{n}x^{4s}_j}}{||\vec{a}||^{4s+1}+||\vec{b}||^{4s+1}}\bigg)\bigg| dx_1dx_2\cdots dx_n \geq \frac{\bigg|\prod_{j=1}^{n}|b_j|-|a_j|\bigg|}{\log \log (\langle a,b\rangle)}\nonumber
\end{align}
for all $s\in \mathbb{N}$, where $\langle,\rangle$ denotes the inner product and $i^2=-1$.
\end{theorem}
\bigskip

\begin{theorem}
Let $\vec{a}=(a_1,a_2,\ldots,a_n),\vec{b}=(b_1,b_2,\ldots,b_n)\in \mathbb{R}^n$ be such that $1<\langle \vec{a},\vec{b} \rangle$. We have
\begin{align}
\bigg |\int \limits_{|a_n|}^{|b_n|} \int \limits_{|a_{n-1}|}^{|b_{n-1}|}\cdots \int \limits_{|a_1|}^{|b_1|}\sqrt[4s]{\sum \limits_{i=1}^{n}x^{4s}_i}dx_1dx_2\cdots dx_n\bigg|&\leq \frac{|\langle \vec{a},\vec{b}\rangle|}{2\pi |\log(\langle \vec{a},\vec{b}\rangle)|}\times (||\vec{a}||^{4s+1}+||\vec{b}||^{4s+1})\nonumber \\&\times \bigg|\prod_{i=1}^{n}|b_i|-|a_i|\bigg|\nonumber
\end{align}
for all $s\in \mathbb{N}$, where $\langle,\rangle$ denotes the inner product.
\end{theorem}

\begin{theorem}
Let $\vec{a}=(a_1,a_2,\ldots,a_n),\vec{b}=(b_1,b_2,\ldots,b_n)\in \mathbb{R}^n$ with $a_j,b_j>0$ and $\langle,\rangle$ denote the inner product such that $1<\langle a,b \rangle \leq e$ with $\langle a,b \rangle \neq 1$. We have
\begin{align}
\bigg|\int \limits_{|a_n|}^{|b_n|} \int \limits_{|a_{n-1}|}^{|b_{n-1}|}\cdots \int \limits_{|a_1|}^{|b_1|}\frac{1}{\sqrt[4s+3]{\sum \limits_{j=1}^{n}x^{4s+3}_j}}dx_1dx_2\cdots dx_n\bigg| \geq \frac{2\pi \times |\log (\langle a,b \rangle)|\bigg|\prod_{j=1}^{n}|b_j|-|a_j|\bigg|}{||\vec{a}||^{4s+4}+||\vec{b}||^{4s+4}}\nonumber.\nonumber
\end{align}
\end{theorem}

\section{The local product and associated space}

In this section, we introduce and study the notion of the \emph{local product} and associated space.

\begin{definition}
Let $\vec{a},\vec{b}\in \mathbb{C}^n$ and $f:\mathbb{C}\longrightarrow \mathbb{C}$ be continuous on $\cup_{j=1}^{n}[|a_j|,|b_j|]$. Let $(\mathbb{C}^n,\langle,\rangle)$ be a complex inner product space. By the $k^{th}$ \emph{local product} of $\vec{a}$ with $\vec{b}$ on the \emph{sheet} $f$, we mean the bi-variate map 
$$
\mathcal{G}^k_f:(\mathbb{C}^n,\langle,\rangle) \times (\mathbb{C}^n,\langle,\rangle) \longrightarrow \mathbb{C}
$$ 
such that 
\begin{align}
\mathcal{G}^k_f(\vec{a};\vec{b})=f(\langle \vec{a},\vec{b}\rangle)\int \limits_{|a_n|}^{|b_n|} \int \limits_{|a_{n-1}|}^{|b_{n-1}|}\cdots \int \limits_{|a_1|}^{|b_1|}f\circ \mathbf{e}\bigg((i)^k \frac{\sqrt[k]{\sum \limits_{j=1}^{n}x^k_j}}{||\vec{a}||^{k+1}+||\vec{b}||^{k+1}}\bigg)dx_1dx_2\cdots dx_n\nonumber
\end{align}
where $\langle,\rangle$ denotes the inner product and where $\mathbf{e}(q)=e^{2\pi iq}$. We denote an inner product space with a $k^{th}$ \emph{local product} defined over a sheet $f$ as the $k^{th}$ local product space over a sheet $f$. We denote this space by the triple $(\mathbb{C}^n,\langle,\rangle,\mathcal{G}^k_f(;))$.
\end{definition}
\bigskip

In certain ways, the $k^{th}$ local product is a universal map induced by a sheet. In other words, a local product can be made by carefully selecting the sheet. We get the local product by making our sheet the constant function $f:=1$

\begin{align}
\mathcal{G}^{k}_1(\vec{a};\vec{b})&=\int \limits_{|a_n|}^{|b_n|} \int \limits_{|a_{n-1}|}^{|b_{n-1}|}\cdots \int \limits_{|a_1|}^{|b_1|}dx_1dx_2\cdots dx_n\nonumber \\&=\prod_{i=1}^{n}|b_i|-|a_i|.\nonumber
\end{align}
Similarly, if we take our sheet as $f=\log$, then under the condition that $\langle\vec{a},\vec{b}\rangle \neq 0$, we obtain the induced local product 
\begin{align}
\mathcal{G}^k_{\log}(\vec{a};\vec{b})&=2\pi \times (i)^{k+1}\frac{\log(\langle \vec{a},\vec{b}\rangle)}{||\vec{a}||^{k+1}+||\vec{b}||^{k+1}}\int \limits_{|a_n|}^{|b_n|} \int \limits_{|a_{n-1}|}^{|b_{n-1}|}\cdots \int \limits_{|a_1|}^{|b_1|}\sqrt[k]{\sum \limits_{j=1}^{n}x^k_j}dx_1dx_2\cdots dx_n.\nonumber
\end{align}
Taking the sheet $f=\mathrm{Id}$ to be the identity function, we obtain the associated local product 
\begin{align}
\mathcal{G}^k_{\mathrm{Id}}(\vec{a};\vec{b})=\langle \vec{a},\vec{b}\rangle \int \limits_{|a_n|}^{|b_n|} \int \limits_{|a_{n-1}|}^{|b_{n-1}|}\cdots \int \limits_{|a_1|}^{|b_1|}\mathbf{e}\bigg(\frac{(i)^k\sqrt[k]{\sum \limits_{j=1}^{n}x^k_j}}{||\vec{a}||^{k+1}+||\vec{b}||^{k+1}}\bigg)dx_1dx_2\cdots dx_n.\nonumber
\end{align}
Again, by taking the sheet $f=\mathrm{Id}^{-1}$ with $\langle a,b \rangle \neq 0$, we obtain the corresponding induced $k^{th}$ local product 
\begin{align}
\mathcal{G}^k_{\mathrm{Id}^{-1}}(\vec{a};\vec{b})=\frac{1}{\langle \vec{a},\vec{b}\rangle}\int \limits_{|a_n|}^{|b_n|} \int \limits_{|a_{n-1}|}^{|b_{n-1}|}\cdots \int \limits_{|a_1|}^{|b_1|}\mathbf{e}\bigg(-\frac{(i)^k\sqrt[k]{\sum \limits_{j=1}^{n}x^k_j}}{||\vec{a}||^{k+1}+||\vec{b}||^{k+1}}\bigg)dx_1dx_2\cdots dx_n. \nonumber
\end{align}
Also, by taking the sheet $f=\log \log$, we have the associated $k^{th}$ local product 
\begin{align}
\mathcal{G}^{k}_{\log \log}(\vec{a};\vec{b})&=\log \log (\langle \vec{a},\vec{b}\rangle)\int \limits_{|a_n|}^{|b_n|}\int \limits_{|a_{n-1}|}^{|b_{n-1}|}\cdots\int \limits_{|a_1|}^{|b_1|}\log \bigg(i\frac{\sqrt[k]{\sum \limits_{j=1}^{n}x^{k}_j}}{||\vec{a}||^{k+1}+||\vec{b}||^{k+1}}\bigg)dx_1dx_2\cdots dx_n.\nonumber
\end{align}

\section{Properties of the local product product}

In this section, we study some properties of the local product on a fixed sheet.

\begin{proposition}\label{key properties}
The following holds
\begin{enumerate}
\item [(i)] If $f$ is linear such that $\langle a,b\rangle=-\langle b,a\rangle$ then 
\begin{align}
\mathcal{G}^k_f(\vec{a};\vec{b})=(-1)^{n+1}\mathcal{G}^k_f(\vec{b};\vec{a}).\nonumber
\end{align}

\item [(ii)] Let $f,g:\mathbb{R}\longrightarrow \mathbb{R}^{+}$ be such that $f(t)\leq g(t)$ for any $t\in [1,\infty)$ with $f(<a,b>)<g(<a,b>)$. We have $|\mathcal{G}_f(\vec{a};\vec{b})|\leq |\mathcal{G}_g(\vec{a};\vec{b})|$.
\end{enumerate}
\end{proposition}

\begin{proof}
\begin{enumerate}
\item [(i)] By the linearity of $f$, we can write
\begin{align}
\mathcal{G}^k_f(\vec{a};\vec{b})&=f(\langle \vec{a},\vec{b}\rangle)\int \limits_{|a_n|}^{|b_n|} \int \limits_{|a_{n-1}|}^{|b_{n-1}|}\cdots \int \limits_{|a_1|}^{|b_1|}f\circ \mathbf{e}\bigg((i)^k \frac{\sqrt[k]{\sum \limits_{j=1}^{n}x^k_j}}{||\vec{a}||^{k+1}+||\vec{b}||^{k+1}}\bigg)dx_1dx_2\cdots dx_n \nonumber \\&=f(\langle \vec{a},\vec{b}\rangle)\int \limits_{|a_n|}^{|b_n|} \int \limits_{|a_{n-1}|}^{|b_{n-1}|}\cdots \int \limits_{|a_1|}^{|b_1|}f\circ \mathbf{e}\bigg((i)^k \frac{\sqrt[k]{\sum \limits_{j=1}^{n}x^k_j}}{||\vec{a}||^{k+1}+||\vec{b}||^{k+1}}\bigg)dx_1dx_2\cdots dx_n \nonumber \\&=f(-\langle b,a \rangle)(-1)^n\int \limits_{|b_n|}^{|a_n|} \int \limits_{|b_{n-1}|}^{|a_{n-1}|}\cdots \int \limits_{|b_1|}^{|a_1|}f\circ \mathbf{e}\bigg((i)^k \frac{\sqrt[k]{\sum \limits_{j=1}^{n}x^k_j}}{||\vec{a}||^{k+1}+||\vec{b}||^{k+1}}\bigg)dx_1dx_2\cdots dx_n \nonumber \\&=(-1)^{n+1}f(\langle b,a \rangle)\int \limits_{|b_n|}^{|a_n|} \int \limits_{|b_{n-1}|}^{|a_{n-1}|}\cdots \int \limits_{|b_1|}^{|a_1|}f\circ \mathbf{e}\bigg((i)^k \frac{\sqrt[k]{\sum \limits_{j=1}^{n}x^k_j}}{||\vec{a}||^{k+1}+||\vec{b}||^{k+1}}\bigg)dx_1dx_2\cdots dx_n \nonumber \\&=(-1)^{n+1}\mathcal{G}^k_f(\vec{b};\vec{a}).\nonumber
\end{align}
\bigskip

\item [(ii)] Property (ii) follows very easily from the inequality $f(t)\leq g(t)$.
\end{enumerate} 
\end{proof}

\section{Applications of the local product}

In this section, we explore some applications of the local product. 

\begin{theorem}\label{app3}
Let $\vec{a}=(a_1,a_2,\ldots,a_n),\vec{b}=(b_1,b_2,\ldots,b_n)\in \mathbb{R}^n$ be such that $e^e<\langle \vec{a},\vec{b} \rangle$ and $b_j>a_j$ for $1\leq j\leq n$. We have
\begin{align}
\int \limits_{|a_n|}^{|b_n|} \int \limits_{|a_{n-1}|}^{|b_{n-1}|}\cdots \int \limits_{|a_1|}^{|b_1|}\bigg|\log \bigg(i\frac{\sqrt[4s]{\sum \limits_{j=1}^{n}x^{4s}_j}}{||\vec{a}||^{4s+1}+||\vec{b}||^{4s+1}}\bigg)\bigg| dx_1dx_2\cdots dx_n\geq \frac{\bigg|\prod_{j=1}^{n}|b_j|-|a_j|\bigg|}{\log \log (\langle a,b\rangle)}\nonumber
\end{align}
for all $s\in \mathbb{N}$, where $\langle,\rangle$ denotes the inner product and $i^2=-1$.
\end{theorem}

\begin{proof}
Let $f:\mathbb{R} \longrightarrow \mathbb{R}^{+}$ and $\vec{a},\vec{b}\in \mathbb{R}^n$ be such that $e^e< \langle\vec{a},\vec{b}\rangle$. We note that 
\begin{align}
\mathcal{G}^{4s}_{\log \log}(\vec{a};\vec{b})&=\log \log (\langle \vec{a},\vec{b}\rangle)\int \limits_{|a_n|}^{|b_n|} \int \limits_{|a_{n-1}|}^{|b_{n-1}|}\cdots \int \limits_{|a_1|}^{|b_1|}\log \bigg(i\frac{\sqrt[4s]{\sum \limits_{j=1}^{n}x^{4s}_j}}{||\vec{a}||^{4s+1}+||\vec{b}||^{4s+1}}\bigg)dx_1dx_2\cdots dx_n\nonumber
\end{align}
taking $k=4s$ for any $s\in \mathbb{N}$. Also, by taking the sheet $f:=1$ as a constant function, we obtain the associated local product
\begin{align}
\mathcal{G}^{4s}_{1}(\vec{a};\vec{b})&=\int \limits_{|a_n|}^{|b_n|} \int \limits_{|a_{n-1}|}^{|b_{n-1}|}\cdots \int \limits_{|a_1|}^{|b_1|}dx_1dx_2\cdots dx_n\nonumber \\&=\prod_{i=1}^{n}|b_i|-|a_i|.\nonumber
\end{align}
Since $|\log it|=|\log t+i\frac{\pi}{2}|\geq 1$ on $\mathbb{R}^{+}$, the inequality is a consequence by using Proposition \ref{key properties}.
\end{proof}

\begin{theorem}\label{app2}
Let $\vec{a}=(a_1,a_2,\ldots,a_n),\vec{b}=(b_1,b_2,\ldots,b_n)\in \mathbb{R}^n$ be such that $1<\langle \vec{a},\vec{b} \rangle$. We have
\begin{align}
\bigg |\int \limits_{|a_n|}^{|b_n|}\int \limits_{|a_{n-1}|}^{|b_{n-1}|}\cdots \int \limits_{|a_1|}^{|b_1|}\sqrt[4s]{\sum \limits_{i=1}^{n}x^{4s}_i}dx_1dx_2\cdots dx_n\bigg|&\leq \frac{|\langle \vec{a},\vec{b}\rangle|}{2\pi |\log(\langle \vec{a},\vec{b}\rangle)|}\times (||\vec{a}||^{4s+1}+||\vec{b}||^{4s+1})\nonumber \\&\times \bigg|\prod_{i=1}^{n}|b_i|-|a_i|\bigg|\nonumber
\end{align}
for all $s\in \mathbb{N}$, where $\langle,\rangle$ denotes the inner product.
\end{theorem}

\begin{proof}
Let $f:\mathbb{R} \longrightarrow \mathbb{R}^{+}$ and $\vec{a},\vec{b}\in \mathbb{R}^n$ be such that $1<\langle\vec{a},\vec{b}\rangle$. We note that 
\begin{align}
\mathcal{G}^{4s}_{\log}(\vec{a};\vec{b})&=2\pi \times (i)^{4s+1}\frac{\log(\langle \vec{a},\vec{b}\rangle)}{||\vec{a}||^{4s+1}+||\vec{b}||^{4s+1}}\int \limits_{|a_n|}^{|b_n|} \int \limits_{|a_{n-1}|}^{|b_{n-1}|}\cdots \int \limits_{|a_1|}^{|b_1|}\sqrt[4s]{\sum \limits_{j=1}^{n}x^{4s}_j}dx_1dx_2\cdots dx_n\nonumber
\end{align}
taking $k=4s$ for any $s\in \mathbb{N}$. Also, by taking the sheet $f:=|\cdot|$ as the distance function, we obtain in this setting the associated local product
\begin{align}
\mathcal{G}^{4s}_{|\cdot|}(\vec{a};\vec{b})&=|\langle \vec{a},\vec{b}\rangle|\int \limits_{|a_n|}^{|b_n|} \int \limits_{|a_{n-1}|}^{|b_{n-1}|}\cdots \int \limits_{|a_1|}^{|b_1|}dx_1dx_2\cdots dx_n\nonumber \\&=|\langle \vec{a},\vec{b}\rangle|\times \bigg|\prod_{i=1}^{n}|b_i|-|a_i|\bigg|.\nonumber
\end{align}
Since $\log<|\cdot|$ on $(1,\infty))$, the inequality is a consequence by using Proposition \ref{key properties}.
\end{proof}

\begin{theorem}\label{app1}
Let $\vec{a}=(a_1,a_2,\ldots,a_n),\vec{b}=(b_1,b_2,\ldots,b_n)\in \mathbb{R}^n$ with $a_j,b_j>0$ and $\langle,\rangle$ denote the inner product such that $0<\langle a,b \rangle \leq e$ with $\langle a,b \rangle \neq 1$. We have
\begin{align}
\bigg|\int \limits_{|a_n|}^{|b_n|} \int \limits_{|a_{n-1}|}^{|b_{n-1}|}\cdots \int \limits_{|a_1|}^{|b_1|}\frac{1}{\sqrt[4s+3]{\sum \limits_{j=1}^{n}x^{4s+3}_j}}dx_1dx_2\cdots dx_n\bigg| \geq \frac{2\pi \times |\log (\langle a,b \rangle)|\bigg|\prod_{j=1}^{n}|b_j|-|a_j|\bigg|}{||\vec{a}||^{4s+4}+||\vec{b}||^{4s+4}}\nonumber.\nonumber
\end{align}
\end{theorem}

\begin{proof}
Let $f:\mathbb{R} \longrightarrow \mathbb{R}^{+}$ and $\vec{a},\vec{b}\in \mathbb{R}^n$. We note that 
\begin{align}
\mathcal{G}^{k}_{\frac{1}{\log}}(\vec{a};\vec{b})&=\frac{1}{\log (\langle a,b \rangle)}\times (||\vec{a}||^{k+1}+||\vec{b}||^{k+1})\times \frac{1}{2i^{k+1}\pi}\nonumber \\& \times \int \limits_{|a_n|}^{|b_n|} \int \limits_{|a_{n-1}|}^{|b_{n-1}|}\cdots \int \limits_{|a_1|}^{|b_1|}\frac{1}{\sqrt[k]{\sum \limits_{j=1}^{n}x^k_j}}dx_1dx_2\cdots dx_n\nonumber
\end{align}
taking $k=4s+3$ for any $s\in \mathbb{N}$. Also, by taking the sheet $f:=1$ as a constant function, we obtain the associated local product
\begin{align}
\mathcal{G}^{4s}_{1}(\vec{a};\vec{b})&=\int \limits_{|a_n|}^{|b_n|} \int \limits_{|a_{n-1}|}^{|b_{n-1}|}\cdots \int \limits_{|a_1|}^{|b_1|}dx_1dx_2\cdots dx_n\nonumber \\&=\prod_{i=1}^{n}|b_i|-|a_i|.\nonumber
\end{align}
Since $\frac{1}{\log}\geq 1$ on $(1,e)$, the inequality is a consequence by using Proposition \ref{key properties} and the requirement $0<\langle a,b \rangle \leq e$ with $\langle a,b \rangle \neq 1$.
\end{proof}

\begin{remark}
The notion of a local product on a sheet has been carefully exploited in this paper to prove some new inequalities. In principle, these inequalities could be proved directly without using the notion of the local product on a sheet. However, a direct proof may be challenging and may possibly not be attainable. The notion of a local product is important by itself, as it allows us to examine the interaction of the behaviour of varied functions (sheets) freely chosen in relation to the product with appropriate supports.
\end{remark}
\rule{100pt}{1pt}

\footnote{
\par
.}%

\bibliographystyle{amsplain}

\end{document}